\documentclass[11pt]{amsart}
\newcommand{\ben}{\begin{enumerate}}
\newcommand{\een}{\end{enumerate}}
\newcommand{\eq}[2][label]{\begin{equation}\label{#1}#2\end{equation}}
\newcommand{\av}[2]{\langle #1\rangle_{_{\scriptstyle #2}}}

\newcommand{\ve}{\varepsilon}
\newcommand{\eps}{\varepsilon}

\renewcommand{\phi}{\varphi}
\usepackage{amsmath}
\usepackage{amsfonts}
\usepackage{amssymb}
\usepackage{amsthm}
\usepackage{graphicx,color,hyperref}
\usepackage{mathtools}
\usepackage[normalem]{ulem} 

\usepackage{enumitem}
\usepackage{url}

\usepackage{mathabx}

\usepackage{stmaryrd}

\newcommand{\BMO}{{\rm BMO}}

\newtheorem{theorem}{Theorem}[section]

\newtheorem{lemma}[theorem]{Lemma}
\newtheorem{corollary}[theorem]{Corollary}

\newtheorem*{theorem*}{Theorem}{\bf}{\it}
\newtheorem*{proposition*}{Proposition}{\bf}{\it}
\newtheorem*{observation*}{Observation}{\bf}{\it}
\newtheorem*{lemma*}{Lemma}{\bf}{\it}
\newtheorem*{conjecture*}{Conjecture}{\bf}{\it}

\theoremstyle{definition}

\theoremstyle{remark}
\newtheorem{remark}[theorem]{Remark}

\numberwithin{equation}{section}

\allowdisplaybreaks

\renewcommand{\phi}{\varphi}
\newcommand{\qq}{Q}
\newcommand{\vv}{V}
\newcommand{\ww}{W}
\newcommand{\cc}{\mathcal{C}}
\newcommand{\tr}{Tr}
\newcommand{\cbmo}[1]{\vvvert #1\vvvert}
\newcommand{\catwo}[1]{\llbracket#1\rrbracket_{A_2}}

\title{Monotone rearrangement in averaging classes}
\begin{document}
\author{Marat Abdrakhmanov}
\author{Leonid Slavin}
\author{Pavel Zatitskii}
\address{St. Petersburg State University, St. Petersburg, 199178, Russia}

\email{st067897@student.spbu.ru}

\address{University of Cincinnati, P.O. Box 210025, OH 45221-0025, USA}

\email{leonid.slavin@uc.edu}

\address{University of Cincinnati, P.O. Box 210025, OH 45221-0025, USA}

\email{zatitspl@ucmail.uc.edu}

\thanks{The second author is supported by the Simons Foundation, collaboration grants 317925 and 711643}
\thanks{The second and third authors express thanks to the Taft Research Foundation at the University of Cincinnati}



\begin{abstract}
We consider a general collection of function classes on the interval $[0,1]$ defined in terms of certain averages and show that monotone rearrangement does not increase the class constant in each case. The formulation includes $\BMO$ and $A_2$ with a special choice of the norm and, respectively, characteristic.
\end{abstract}

\maketitle

\section{Introduction}
Throughout this note, we write $I=[0,1];$ $J$ and $L$ will denote subintervals of $I;$ $\av{\,\cdot\,}J$ will denote the Lebesgue average over $J,$ $\av{\varphi}J=\frac1{|J|}\,\int_J\varphi.$ 

Let $\qq\colon \mathbb{R} \to [0,+\infty)$ be a convex, even function; note that $\qq$ is thus increasing on $[0,+\infty)$.

Let $\phi$ be a measurable real-valued function on $I$. For a subinterval $J\subset I$ and $c\in\mathbb{R},$ let  
\eq[0.1]{
\vv_c(\phi, J) = \av{\qq(\phi-c)}J,
}
\eq[0.2]{
\vv(\phi,J) = \inf \{\vv_c(\phi,J)\colon c \in \mathbb{R}\},
}
and
\eq[0.3]{
\ww(\phi,J) = \sup\{\vv(\phi,J)\colon L \subset J\}.
}
Our main result concerns the class of all $\varphi$ on $I$ such that $W(\varphi,I)<\infty;$ we will call such classes {\it averaging}. In general, they are different for different $\qq.$ As an illustration, we consider two specific choices of $\qq.$ 
If $\qq(t) = |t|^p$ for $p\in[1,\infty),$ then $\ww(\phi,I)<\infty$ if and only if $\phi\in\BMO(I).$ In this case,
$\ww(\phi,I)$ is the $p$-th power of an equivalent BMO norm of $\phi,$ which we denote by $\cbmo{\cdot}_p$:
\eq[0.9]{
\cbmo{\varphi}_p=\sup_{J\subset I}\inf_c\,\av{\big|\varphi-c\big|^p}J^{1/p}.
}
This is a variant of the usual $\BMO^p$ norm where, having fixed an interval $J,$ one sets $c=\av{\varphi}J$ instead of taking infimum in $c$:
\eq[1]{
\|\varphi\|_{p}=\sup_{J\subset I}\,\av{\big|\varphi-\av{\varphi}J\big|^p}J^{1/p}.
}
Note that $\cbmo{\cdot}_p$ and $\|\cdot\|_{p}$ coincide for $p=2;$ for other values of $p$ we have elementary equivalence:
$$
\cbmo{\varphi}_p\le \|\varphi\|_p\le 2\cbmo{\varphi}_p.
$$

If $\qq(t)=e^{|t|},$ then $\ww(\varphi,I)<\infty$ if and only if $e^\varphi\in A_2.$ We denote the corresponding $A_2$-characteristic by $\catwo{w}$:
$$
\catwo{w}=\sup_{J\subset I}\inf_c\,\av{e^{|\log w-c|}}J.
$$
Recall that the usual $A_2$-characteristic is given by $[w]_{A_2}=\sup_{J\subset I}\av{w}J\av{w^{-1}}J.$ To see the relationship between the two characteristics, note that on one hand,
$$
\av{w}J\av{w^{-1}}J=\av{e^{\log w-c}}J\av{e^{-(\log w-c)}}J\le 
\Big(\av{e^{|\log w-c|}}J\Big)^2.
$$
On the other hand, since for every $J\subset I$ we have $\av{e^{\pm(\log w-\av{\log w}J)}}J\le [w]_{A_2},$ 
$$
\inf_c\av{e^{|\log w-c|}}J\le
\av{e^{|\log w-\av{\log w}J|}}J
\le \av{e^{\log w-\av{\log w}J}}J+\av{e^{-(\log w-\av{\log w}J)}}J\le 2[w]_{A_2}.
$$
Altogether, we have
$$
\sqrt{[w]_{A_2}}\le \catwo{w}\le 2[w]_{A_2}.
$$
\begin{remark}
One could also use a different choice of $Q$ to describe $A_2,$ e.\,g. $Q(t)=\cosh t.$
\end{remark}

An important question regarding averaging classes such as BMO and $A_2$ is whether they are invariant under monotone rearrangement and, furthermore, whether  the class constant is preserved. 
As a practical matter, such preservation allows one to deal only with monotone elements of the class when proving integral inequalities, making analysis easier.
Recall that for a function $f\colon I\to \mathbb{R},$ its decreasing rearrangement is the function $f^*$ on $I$ given by
$$
f^*(t)=\inf\{\lambda\colon |\{s\in I\colon f(s)>\lambda\}|\le t\}.
$$
Note that for a weight $w$ on $I,$ we have
$$
w^*=e^{(\log w)^*}.
$$

In~\cite{Klemes}, Klemes proved that rearrangement does not increase the $\BMO^1$ norm given by~\eqref{1} with $p=1,$ i.\,e., that if $\varphi\in\BMO(I),$ then $\varphi^*\in \BMO(I)$ and 
$$
\|\varphi^*\|_{1}\le \|\phi\|_{1}.
$$
In the same paper, he also claimed that a similar approach could be used to show that for all $p>1$ the $\BMO^p$ norm does not increase under rearrangement. We have not been able to extend Klemes's argument to that range of $p.$ However, we have been able to prove a preservation-of-constant result in the equivalent formulation involving taking infimum over all~$c.$ (Klemes's argument does not extend to this case either.) In fact, we show that a general functional $W$ given by \eqref{0.1}-\eqref{0.3} does not increase under rearrangement. Here is our main theorem.
\begin{theorem}\label{th1}
Let $\phi$ be a measurable real-valued function on $I.$ Then
\eq[eqmain]{
\ww(\phi^*,I) \leq \ww(\phi,I).
}

Consequently, if $\varphi\in\BMO(I),$ then $\varphi^*\in\BMO(I)$ and for all $p\ge1,$
$$
\cbmo{\varphi^*}_p\le \cbmo{\varphi}_p.
$$
Furthermore, if $w\in A_2(I),$ then $w^*\in A_2(I)$ and
$$
\catwo{w^*}\le\catwo{w}.
$$
\end{theorem}
To elaborate further on the novelty of the result, we note, without going into details, that in~\cite{sz} it is shown that rearrangement does not increase the class constant in any function class $A_\Omega$ on $I$ generated in a canonical fashion by a given non-convex set $\Omega$ in the plane. This statement covers $\BMO$ with the $\|\cdot\|_2$ norm and Muckenhoupt weights $A_p,$ $1< p\le\infty,$ with the usual choice of characteristic. However, none of the classes defined using \eqref{0.1}-\eqref{0.3} (other than $\BMO^2$) fall under this formulation. That said, some elements of our proof of Theorem~\ref{th1} are appropriate variants of those in~\cite{sz} and related literature, and we note the connection as we go along.

In addition to~\cite{Klemes} and~\cite{sz}, let us mention some further recent results about rearrangement in averaging classes. The monograph~\cite{Kor2}
deals with $\BMO^1$ and various weight classes on rectangles. The continuity of the rearrangement operator on BMO and VMO is studied in~\cite{Ryan2}. In~\cite{sz_vmo}, it is proved that 
rearrangement does not increase the Campanato norm in VMO on $[0,1].$ 
In~\cite{Ryan1} and ~\cite{studia} the dimensional behavior of rearrangement is studied in the context of $\BMO^1(\mathbb{R}^n)$ and the dyadic $\BMO^2([0,1]^n),$ respectively. 

We now turn to the proof of Theorem~\ref{th1}. We first prove it under the additional assumption that $Q$ in~\eqref{0.1} is strictly convex and then dispose of that assumption.

\section{Proof for strictly convex $Q$}
\label{str_con}
In this section, we assume that the function $Q$ in~\eqref{0.1} is strictly convex.
Then for any measurable $\phi$ defined on $J\subset I,$ the set $A=\{c \in \mathbb{R}\colon \vv_c(\phi,J)<\infty\}$ is convex, and the function $c \mapsto \vv_c(\phi,J)$ is strictly convex on $A$. If $A \ne \varnothing,$ there exists the unique $\cc \in A$ such that  
\eq[eq1]{
\vv_{\cc}(\phi,J) = \inf \{\vv_c(\phi,J)\colon c \in \mathbb{R}\} = \vv(\phi,J).
}
We denote this unique constant $\cc$ by $\cc(\phi,J)$.

Let $\phi_\pm$ be functions defined on intervals $I_\pm = [a_\pm, b_\pm]$; let $\alpha \in [0,1]$. The $\alpha$-concatenation of $\varphi^\pm$ is a function $\varphi_\alpha$ on $I$ defined as follows: 
$$
\phi_\alpha(s) = 
\begin{cases}
\phi_-(a_- + (b_--a_-)\frac{s}{\alpha}),\quad s \in [0,\alpha),\\
\phi_+(a_+ + (b_+-a_+)\frac{s-\alpha}{1-\alpha}),\quad s \in [\alpha,1].\\
\end{cases}
$$ 
Thus, the restriction of $\phi_\alpha$ to $[0,\alpha]$ is the rescaling of $\phi_-$ from $I_-,$ and the restriction of $\phi_\alpha$ to $[\alpha,1]$ is the rescaling of $\phi_+$ from $I_+$. It is easy to see that for any $c \in \mathbb{R}$ we have
$$
\vv_c(\phi_\alpha, I) = \alpha \vv_c(\phi_-,I_-) + (1-\alpha) \vv_c(\phi_+,I_+).
$$ 
From this formula it is clear that for $\alpha \in (0,1)$ the left-hand side is finite if and only if both summands in the right-hand side are finite.

\begin{remark}\label{rem2}
The function $(c,\alpha) \mapsto \vv_c(\phi_\alpha, I)$ is finite on some subset of $\mathbb{R}\times [0,1]$ and continuous on it. This function is linear with respect to $\alpha$ and strictly convex with respect to $c$. If there is some $c \in \mathbb{R}$ such that both $\vv_c(\phi_+,I_+)<\infty$ and $\vv_c(\phi_-,I_-)<\infty,$ then the function $\alpha \mapsto
\cc(\phi_\alpha,I)$ is correctly defined and continuous on $[0,1]$. 
\end{remark}
The following lemma is an adaption to the present setting of a classical splitting argument due to V.~Vasyunin, which first appeared in \cite{vas1}.
 
\begin{lemma}[Splitting Lemma]\label{lem1}
Let $0<\tilde\ve<\ve,$ $0 < \delta < \min(\frac{1}{2}, 1 - \frac{\qq(\tilde\ve)}{\qq(\ve)})$. If $\phi$ is a function on an interval $J=[a,b]$ such that $\ww(\phi,J) \leq \qq(\tilde\ve),$ then there is $t \in (\delta, 1-\delta)$ with the following property: if $J_- =[a,(1-t)a+tb]$ and $J_+ = [(1-t)a+tb,b],$ and functions $\phi_\pm$ are restrictions of  $\phi$ to $J_\pm,$ respectively, then 
\eq[eq8]{
\vv(\phi_\alpha, I) \leq \qq(\eps), \qquad \text{ for all } \alpha \in [0,1].
}
\end{lemma}
\begin{proof}
Let $\cc = \cc(\phi,J)$.
We note that for any choice of $t\in(0,1),$ 
$$
|J_-| \vv_\cc(\phi_-, J_-) + |J_+| \vv_\cc(\phi_+, J_+) = |J|\vv_\cc(\phi, J) \leq |J|\qq(\tilde\eps).
$$
Therefore 
\eq[eq3]{
\vv_\cc(\phi_\pm, J_\pm) \leq \frac{|J|}{|J_\pm|}\,\qq(\tilde\eps). 
}
If $t = \delta,$ then 
$$
|J_+| = (1 - \delta)|J| > \frac{\qq(\tilde\eps)}{\qq(\eps)}\,|J|,
$$
and, from~\eqref{eq3}, 
\eq[eq4]{
\vv_\cc(\phi_+, J_+) \leq\frac{1}{1-\delta}\, \qq(\tilde\eps) < \qq(\eps).
}
Similarly, if $t = 1-\delta,$ then 
\eq[eq5]{
\vv_\cc(\phi_-, J_-) < \qq(\eps).
}

Let $\Psi(t,\alpha) = \vv_\cc(\phi_\alpha, I)$. Note that if $\alpha=t,$ then $\phi_\alpha$ is a rescaling of $\phi$ to $I$; therefore,
$$
\Psi(t,t) =  \vv_\cc(\phi, J) \leq \qq(\tilde\eps).
$$
The function $\Psi$ is clearly continuous on $[\delta, 1-\delta]\times [0,1]$ and linear in $\alpha$ for any fixed $t$. From~\eqref{eq4} and~\eqref{eq5} we have
$$
\Psi(\delta,0)<\qq(\eps), \qquad \Psi(1-\delta,1)<\qq(\eps).
$$
It follows that there is some $t \in [\delta,1-\delta]$ such that $\Psi(t,\alpha) \leq \qq(\eps)$ for all $\alpha \in [0,1]$. Indeed, 
if $t=\delta$ does not fit, then $\Psi(\delta,1)> \qq(\eps)>\Psi(1-\delta,1)$. By continuity, we can find $t \in [\delta, 1-\delta]$ such that 
$\Psi(t,1) = \qq(\eps)$.  For this choice of $t$ we have $\Psi(t,t) < \qq(\eps),$ thus by linearity with respect to $\alpha,$ we have $\Psi(t,\alpha) \leq \qq(\eps)$ for all $\alpha \in [0,1]$.
\end{proof}

We now introduce a suitable analog of the Bellman function that was used in~\cite{sz} to prove the rearrangement result for classes $A_\Omega.$ This function is shown to be locally concave in a certain sense (Lemma~\ref{lem2}), which allows a  variant of the classical Bellman induction (Lemma~\ref{lem23}).

Let $\phi$ be a function on an interval $J$. Define
\eq[eq180504]{
G(\phi) = \begin{cases}
0, &  \qq(\eps) \geq \vv_\eps(\phi,J) \text{ or } \cc(\phi,J) \geq \eps,\\
\qq(\eps) - \vv_\eps(\phi,J), &  \qq(\eps) < \vv_\eps(\phi,J) \text{ and } \cc(\phi,J) < \eps.
\end{cases}
}
We note that always $G \leq 0,$ and that if 
\eq[eq6]{
G(\phi) > \qq(\eps) - \vv_\eps(\phi,J), \text{ then }  \cc(\phi,J) \geq \eps.
}

\begin{lemma}\label{lem2}
Let $\phi_\pm$ be functions on arbitrary intervals $I_\pm$ such that $$
\vv(\phi_\alpha, I) \leq \qq(\eps)
$$
for all $\alpha \in [0,1]$. Then, for any $\alpha \in [0,1]$ we have
$$
G(\phi_\alpha) \geq \alpha G(\phi_-) + (1-\alpha) G(\phi_+).
$$
\end{lemma}
\begin{proof}
Let $\alpha \in (0,1)$. If $G(\phi_\alpha) = 0,$ there is nothing to prove. Thus, we suppose $G(\phi_\alpha) = \qq(\eps) - \vv_\eps(\phi_\alpha,I) < 0$. Then, by the definition of $G$ we have $\cc(\phi_\alpha,I) < \eps$. 
If 
\eq[eq7]{
G(\phi_-) \leq \qq(\eps) - \vv_\eps(\phi_-,I_-)\quad\text{and}\quad
G(\phi_+) \leq \qq(\eps) - \vv_\eps(\phi_+,I_+),
}
then
$$
\alpha G(\phi_-) + (1-\alpha) G(\phi_+) \leq \qq(\eps) - \alpha\vv_\eps(\phi_-,I_-) - (1-\alpha)\vv_\eps(\phi_+,I_+) = \qq(\eps) - \vv_\eps(\phi_\alpha,I) = G(\phi_\alpha) .
$$

Assume now that one of the inequalities in~\eqref{eq7} does not hold, say 
$$
G(\phi_-) > \qq(\eps) - \vv_\eps(\phi_-,I_-).
$$
Then, by~\eqref{eq6} we have $\cc(\phi_-,I_-) \geq \eps > \cc(\phi_\alpha,I)$. By continuity, we can find $\beta \in (\alpha, 1]$ such that 
$\cc(\phi_\beta,I) = \eps$. Then,
\begin{align*}
\vv(\phi_\beta,I) &= \vv_\eps(\phi_\beta,I) = \beta \vv_\eps(\phi_-,I_-) + (1-\beta) \vv_\eps(\phi_+,I_+)\\
&=\frac{1-\beta}{1-\alpha}\,\vv_\eps(\phi_\alpha,I) + \frac{\beta - \alpha}{1-\alpha} \,\vv_\eps(\phi_-,I_-)  > \qq(\eps),
\end{align*}
which contradicts the hypothesis of the lemma.
\end{proof}

\begin{lemma}\label{rem3}
Let $\phi$ be a bounded function on $I$. Then, for almost all $s \in I$ and any sequence of subintervals $J_n$ of $I$ containing $s$ and such that $|J_n| \to 0,$ we have, for any $c \in \mathbb{R},$ 
\begin{gather}
\label{eq180501}
\vv_c(\phi, J_n) \to \qq(\phi(s) - c),\\
\label{eq180502}
\cc(\phi, J_n) \to \phi(s).
\end{gather}
Therefore, 
\eq[eq180503]{
G(\phi\big|_{_{J_n}}) \to 
\begin{cases}
0, & \phi(s) \geq 0,\\
\qq(\eps) - \qq(\phi(s)-\eps),& \phi(s)<0. 
\end{cases}
}
\end{lemma}
\begin{proof}
For any $c\in\mathbb{R}$ fixed~\eqref{eq180501} holds for almost any $s \in I$ by the Lebesgue Differentiation Theorem. Therefore, for all $s$ from a subset of full measure in $I,$ relation~\eqref{eq180501} holds for all $c \in \mathbb{Q}$. For any such $s$ and a sequence of subintervals $J_n$ fixed, the sequence of strictly convex functions $F_n\colon c \mapsto \vv_c(\phi, J_n)$ converges to a strictly convex function $F\colon c \mapsto \qq(\phi(s) - c)$ on $\mathbb{Q}$. It follows from the convexity of $Q$ that $F_n$ converges to $F$ pointwise on $\mathbb{R}$ and uniformly on compact subsets. Moreover, since $F(\pm\infty)=\infty,$ we see that the points of minimum of $F_n$ converge to the point of minimum of $F,$ which is exactly~\eqref{eq180502}. Relation~\eqref{eq180503} follows from 
\eqref{eq180501} with $c=\eps,$ \eqref{eq180502}, and the definition~\eqref{eq180504} of the function~$G$; it can be checked directly for each of the four cases: $\phi(s)<0,$ $\phi(s)=0,$ $0< \phi(s)\le\eps,$ and $\phi(s)>\eps.$
\end{proof}

\begin{lemma}
\label{lem23}
Let $\phi$ be a non-negative bounded function on $J$. If $\ww(\phi,J)<\qq(\eps),$ then either $\cc(\phi,J)\geq\eps$ or $\vv_\eps(\phi,J)\leq \qq(\eps)$.
\end{lemma}
\begin{proof}
It is enough to prove that $G(\phi)=0$.
Find $\tilde\eps \in (0, \eps)$ such that $\ww(\phi,J)<\qq(\tilde\eps)$ and take $\delta$ as in Lemma~\ref{lem1}. That lemma gives two intervals, $J_-$ and $J_+,$ such that~\eqref{eq8} holds and $\frac{|J_\pm|}{|J|}\in (\delta,1-\delta)$. Then, by Lemma~\ref{lem2} we have 
$$
|J|G(\phi) \geq |J_-|G(\phi\big|_{J^-}) +|J_+|G(\phi\big|_{J^+}). 
$$
We repeat this procedure and split each of $J_\pm$ into two subintervals, apply consecutively Lemmas~\ref{lem1} and~\ref{lem2}, and so forth. After $N$ steps we will have a splitting of $J$ into $2^N$ subintervals, say $J^N_{k},$ $k=1,\dots,2^N,$ such that $\frac{|J^N_k|}{|J|}\in (\delta^N, (1-\delta)^N)$ for each $k,$ and
$$
|J|G(\phi) \geq \sum_{k=1}^{2^N} |J^N_k| G(\phi\big|_{J^N_{k}}). 
$$
Since $\phi$ is bounded, the set $\big\{G(\phi|_{J^N_{k}})\big\}_{k,N}$ is bounded. By the dominated convergence theorem and Lemma~\ref{rem3}, the limit of the right-hand side as $N\to\infty$ is zero. Therefore, $G(\phi)\geq 0$.
\end{proof}

Observe that the functionals $\vv_{c}$ and $\cc$ have additive homogeneity: for a function $\varphi$ on $J$ and any $\tau\in\mathbb{R},$
$$
\vv_{\pm c+\tau}(\pm\phi+\tau,J)=\vv_{c}(\phi,J)\quad \text{and}\quad\cc(\pm\varphi+\tau,J)=\pm\cc(\varphi)+\tau.
$$
Applying Lemma~\ref{lem23} to $\varphi-A$ and $B-\varphi,$ we obtain the following corollary.
\begin{corollary}\label{cor1}
Let $A,B\in\mathbb{R},$ $A \leq B$; let $\eps>0$; and let $\phi \colon J \to [A,B]$ satisfy $\ww(\phi, J) < \qq(\eps)$. Then either 
$$
\cc(\phi,J)\geq A+\eps \quad \text{ or } \quad  \vv_{A+\eps}(\phi,J)\leq \qq(\eps).
$$
Similarly, either 
$$
\cc(\phi,J)\leq B-\eps \quad \text{ or } \quad \vv_{B-\eps}(\phi,J)\leq \qq(\eps).
$$
\end{corollary}

\begin{lemma}\label{lem3}
Let $f \colon \mathbb{R} \to \mathbb{R}$ be a 1-Lipschitz function; let $\phi \colon J \to \mathbb{R}$. Then,
$$
\vv(f\circ \phi, J) \leq \vv(\phi,J), \qquad \ww(f\circ \phi, J) \leq \ww(\phi,J).
$$
\end{lemma}
\begin{proof}
For any $\cc \in \mathbb{R}$ and $s \in J$ we have $|f(\phi(s)) - f(\cc)| \leq |\phi(s)-\cc|$. Therefore,
$\qq(f(\phi(s)) - f(\cc))\leq \qq(\phi(s)-\cc)$. Integrate to obtain 
$$
\vv_{f(\cc)}(f\circ \phi,J) \leq \vv_{\cc}(\phi,J), 
$$
and hence
$$
\vv(f\circ \phi, J) \leq \vv(\phi,J).
$$
Therefore, the same relation holds for the functional $\ww$:
$$
\ww(f\circ \phi, J) \leq \ww(\phi,J).
$$
\end{proof}

Let $A,B \in \mathbb{R},$ $A \le B$. Define 
$$
\tr_{A,B}(s) = A\chi_{(-\infty,A)}(s) + s \chi_{[A,B]}(s) + B\chi_{(B,+\infty)}(s), \qquad s \in \mathbb{R}. 
$$
The function $\tr_{A,B}$ is 1-Lipschitz on $\mathbb{R},$ and for any $\phi \colon J \to \mathbb{R}$ the function $\tr_{A,B}\circ \phi$ is the two-sided truncation of $\phi$ at the levels $A$ and $B$:
$$
\tr_{A,B}\circ \phi = \min(B, \max (A, \phi)). 
$$
Lemma~\ref{lem3} implies that the truncation does not increase the ``$\ww$-characteristic'' of a function:
\eq[eq9]{
\ww(\tr_{A,B}\circ \phi,J) \leq \ww(\phi, J). 
}
We are now in a position to prove Theorem~\ref{th1} for the case of a strictly convex $Q.$
\begin{proof}[Proof of Theorem~\ref{th1}]
Take any $\eps>0$ such that $\ww(\phi, I) < \qq(\eps)$. The theorem will be proved if we show that $\vv(\phi^*, J)\leq \qq(\eps)$ for any $J\subset I.$ It is sufficient to take $J=[t_1,t_2]$ with $0<t_1<t_2<1;$ the cases $t_1=0$ and $t_2=1$ follow by taking a limit. 

Take $A = \phi^*(t_2),$ $B = \phi^*(t_1)$; thus, $A \leq B$. If $B-A \leq 2\eps,$ then for $C:= \frac{A+B}{2}$ we have 
$|\phi^*(s) - C|\leq \eps$ and $\qq(\phi^*(s) - C)\leq \qq(\eps)$ for any $s \in J$. Therefore 
$$
\vv(\phi^*,J) \leq \vv_C(\phi^*,J)  = \frac{1}{|J|}\int_J \qq(\phi^*(s) - C) ds \leq \qq(\eps),
$$ proving the claim in this case.

Assume now that $B\geq A +2\eps$. Let $\psi = \tr_{A,B}\circ \phi$. Then, by~\eqref{eq9} we have $\ww(\psi,I)\leq \ww(\phi,I) < \qq(\eps)$. Since 
$\psi^* = \tr_{A,B}\circ \phi^*$ coincides with $\phi^*$ on $J,$ it is enough to prove that $\vv(\psi^*, J)\leq \qq(\eps)$. 

We note that $\cc(\psi, I)= \cc(\psi^*, I) \in [A,B]$. Write $\cc=\cc(\psi, I).$ If $\cc\in [A+\eps, B-\eps],$ then
\begin{align*}
\qq(\eps) >& \vv_\cc(\psi, I) = \vv_\cc(\psi^*, I) = \int_I \qq(\psi^* - \cc) = \\
&t_1 \qq(B-\cc) + \int_J \qq(\psi^* - \cc) + (1-t_2) \qq(\cc-A)\geq \\
&(1-|J|)\qq(\eps) + \int_J \qq(\psi^* - \cc),
\end{align*}
and therefore 
$$
\qq(\eps) > \frac{1}{|J|}\int_J \qq(\psi^* - \cc) = \vv_\cc(\psi^*,J) \geq \vv(\psi^*,J).
$$

If $\cc \in [A, A+\eps),$ we apply the first part of Corollary~\ref{cor1} to $\psi$. We obtain  
\begin{align*}
\qq(\eps) \geq &\vv_{A+\eps}(\psi,I) = \vv_{A+\eps}(\psi^*,I) = \int_I \qq(\psi^* - A - \eps) = \\
&t_1 \qq(B-A-\eps) +  \int_J \qq(\psi^* - A-\eps) + (1-t_2) \qq(-\eps)\geq\\
&(1-|J|) \qq(\eps) + \int_J \qq(\psi^* - A-\eps),
\end{align*}
and therefore 
$$
\qq(\eps) \geq \frac{1}{|J|}\int_J \qq(\psi^* - A - \eps) = \vv_{A+\eps}(\psi^*,J) \geq \vv(\psi^*,J).
$$
Finally, is $\cc\in(B-\eps,B]$ we repeat the above argument, except using the second part of Corollary~\ref{cor1}.
\end{proof}
\section{Proof for the general case}
\label{general}

For $n \in \mathbb{N},$ let $\qq_n(s) = \qq(s) + \frac{1}{n} s^2,$ $s\in\mathbb{R}$. Then each $\qq_n$ is a strictly convex even function on $\mathbb{R},$ and the sequence $\{\qq_n\}$ converges to $\qq$ uniformly on compact sets. We will use the upper index $n$ for $\vv, \vv_c,$ and $\ww$ defined by~\eqref{0.1}, \eqref{0.2}, and \eqref{0.3}, respectively, but with $Q_n$ instead of $Q$.

First, we will prove the statement of Theorem~\ref{th1} for a bounded function $\phi$ on $I$. Thus assume  $|\phi(s)|<M$ for some $M>0$ and all $s \in I$. If $c\in \mathbb{R}$ and $|c|>2M,$ then $|\phi(s)-c|>M$ for all $s \in I,$ therefore for any subinterval $J\subset I$ we have 
$$\vv_c(\phi, J) = \av{\qq(\phi-c)}J \geq \qq(M) \geq \av{\qq(\phi)}J = \vv_0(\phi, J).$$
Hence,
$$
\vv(\phi, J) = \inf_{c \in \mathbb{R}}\vv_c(\phi, J) = \inf_{|c|\leq 2M }\vv_c(\phi, J). 
$$
The same inequality is true for each $n$:
$$
\vv^n(\phi, J)  =\inf_{|c|\leq 2M }\vv^n_c(\phi, J).  
$$
Since $Q_n$ converges to $Q$ uniformly on $[-3M,3M],$ for any $J\subset I$ we have
$$
\vv_c^n(\phi, J)  \to \vv_c(\phi,J),  
$$
uniformly in $c\in[-2M,2M].$
Therefore, 
$$
\vv^n(\phi, J)  \to \vv(\phi, J).
$$
Thus,
$$
\ww^n(\phi,I) \to \ww(\phi,I).
$$
Of course, the same relation holds for $\phi^*,$ that is, 
$
\ww^n(\phi^*,I) \to \ww(\phi^*,I).
$

Now, we use Theorem~\ref{th1} for $Q_n$ and $\phi$ to get
$$
\ww^n(\phi^*,I) \leq \ww^n(\phi,I).
$$
Letting $n\to \infty,$ we prove~\eqref{eqmain}. 

Now, let us drop the assumption that $\phi$ is bounded. Take a subinterval $J=[t_1,t_2] \subset I$ with $0<t_1<t_2<1$. Let 
$A= \phi^*(t_2),$ $B = \phi^*(t_1)$. Apply Theorem~\ref{th1} for $\psi = \tr_{A,B}\circ \phi$ to get
$$
\ww(\psi^*,I) \leq \ww(\psi,I).
$$
Since the restrictions of $\psi^*$ and $\phi^*$ to $J$ coincide, 
we can write
$$
\vv(\phi^*,J) = \vv(\psi^*,J) \leq \ww(\psi^*,I) \leq \ww(\psi,I) \leq \ww(\phi,I), 
$$ 
where the last inequality follows from~\eqref{eq9}. Taking the supremum over all subintervals $J$ yields the statement of the theorem.

\section*{Acknowledgments}
The authors are grateful to D.~Stolyarov for several helpful comments and suggestions.

\end{document}